\documentclass[11pt]{article}

\usepackage{enumerate, amsmath, amsthm, amsfonts, amssymb, xy}
\usepackage{fullpage, graphicx, latexsym, amscd}
\usepackage[margin=1in]{geometry}
\usepackage{hyperref}
\usepackage{algorithm, algorithmic}
\newenvironment{@abssec}[1]{%
     \if@twocolumn
         \section*{#1}%
         \else
        \vspace{.05in}\footnotesize
       \parindent .2in
       {\bfseries #1. }\ignorespaces
     \fi}
     {\if@twocolumn\else\par\vspace{.1in}\fi}
\newenvironment{keywords}{\begin{@abssec}{Key words}}{\end{@abssec}}
\newenvironment{AMS}{\begin{@abssec}{AMS subject
classification}}{\end{@abssec}}

\input xy
\xyoption{all}
\input amssym.def

\newtheorem{theorem}{Theorem}[section]
\newtheorem{proposition}[theorem]{Proposition}
\newtheorem{lemma}[theorem]{Lemma}
\newtheorem{definition}[theorem]{Definition}
\newtheorem{example}[theorem]{Example}
\newtheorem{rmk}[theorem]{Remark}
\newtheorem{corollary}[theorem]{Corollary}

\numberwithin{equation}{section}

\newcommand{\C}{\mathbb{C}}
\newcommand{\N}{\mathbb{N}}

\newcommand{\R}{\mathbb{R}}
\newcommand{\Z}{\bf{Z}}
\newcommand{\strangek}{\mathbb{\kappa}}
\newcommand{\sk}{\strangek}

\renewcommand{\phi}{\varphi}
\renewcommand{\emptyset}{\varnothing}

\makeatletter
\def\Ddots{\mathinner{\mkern1mu\raise\p@
\vbox{\kern7\p@\hbox{.}}\mkern2mu
\raise4\p@\hbox{.}\mkern2mu\raise7\p@\hbox{.}\mkern1mu}}
\makeatother

\newcommand{\Ideal}{\operatorname{Ideal}}
\newcommand{\Span}{\operatorname{span}\;}
\newcommand{\rank}{\operatorname{rank}}
\newcommand{\ex}{\operatorname{ex}}
\newcommand{\co}{\operatorname{corank}}
\newcommand{\bprime}{\operatorname{\mathbb{B}_{\strangek}}}
\newcommand{\calJ}{\operatorname{\mathcal{J}}}
\newcommand{\calP}{\operatorname{\mathcal{P}}}
\newcommand{\calD}{\operatorname{\mathcal{D}}}
\newcommand{\calQ}{\operatorname{\mathcal{Q}}}
\newcommand{\calH}{\operatorname{\mathcal{H}}}
\newcommand{\bks}{{\backslash}}
\newcommand{\lam}{\operatorname{\lambda}}
\newcommand{\Tet}{\operatorname{\Theta}}
\newcommand{\tet}{\operatorname{\theta}}
\newcommand{\Exp}{\operatorname{\Exp}}
\def\inpro#1{\langle #1\rangle}
\def\esp{\epsilon}
\def\B{{\mathbb{B}}}
\def\I{{\mathbb{I}}}
\def\Span{{\rm span}\,}
\def\Exp{{\rm Exp}}
\def\tet{\theta}
\def\least{{}{_\downarrow}}
\def\all{\forall}
\def\Tet{\Theta}

\def\openR{{{\rm I}\kern-.16em {\rm R}}}
\def\b{{\cal V}}
\let\R\openR
\def\eop{{\begin{proof}[]\end{proof}}}
\begin{document}

\title{External zonotopal algebra}
\author{ Nan Li\thanks{Partially supported by the US National Science
Foundation under Grant DMS-0604423.}\\ Department of Mathematics \\
Massachusetts Institute of Technology\\
\and Amos Ron\thanks{Supported by the US National Science Foundation
under Grants DMS-0602837 and DMS-0914986 and by the National
Institute of General Medical Sciences under Grant
NIH-1-R01-GM072000-01.} \\
Department of Mathematics \& \\
Department of Computer Science \\
University of Wisconsin-Madison }
\date{April 10, 2011}
\maketitle

\begin{keywords}
Multivariate polynomials, polynomial ideals, duality, grading, kernels of
differential operators, polynomial interpolation, box splines, zonotopes,
hyperplane arrangements, matroids, graphs, Hilbert series.
\end{keywords}

\begin{AMS}
13F20, 13A02, 16W50, 16W60, 47F05, 47L20, 05B20, 05B35, 05B45, 05C50,
52B05, 52B12, 52B20, 52C07, 52C35, 41A15, 41A63.
\end{AMS}

\begin{abstract}
We provide a general, unified, framework for {\it external zonotopal
algebra}. The approach is critically based on employing simultaneously
the two dual algebraic constructs and invokes the underlying matroidal
and geometric structures in an essential way. This general theory makes
zonotopal algebra an applicable tool for a larger class of polytopes.
\end{abstract}

\section{Introduction}

{\bf General.} The most common methodology for constructing
multivariate splines is via their definition as volume functions. In
this approach, one begins with a linear map, usually a surjection
$$X:\R^N\to\R^n,$$
and continues by restricting this map to a special polyhedron
${\Z}\subset\R^N$. Most relevant to this paper is the theory of {\bf
box splines}, in which $\Z$ is chosen as the unit cube $[0,1]^N$.
Two geometries underscore box spline theory: that of {\it
zonotopes}, and the dual geometry of {\it hyperplane arrangements}.
The theory continues with the association of the two geometries with
corresponding dual algebraic structures, and culminates with a
seamless cohesion of the geometry, the algebra, the spline
function and pertinent combinatorial properties of the map $X$, where
the latter viewed as a linear matroid.

Attempts to extend the aforementioned constructions beyond the
original setup of box spline theory began in the mid 90's and
reached their successful completion in \cite{HR}: that paper
introduced  a three-layer theory that was coined there {\it
zonotopal algebra}, with the original box spline theory occupying
the middle {\it central} layer. Two other algebraic constructions,
over the same pair of dual geometries and related to the same
matroid $X$, were newly introduced in \cite{HR}: an external theory
and an internal theory. Further developments of zonotopal algebra
were recorded in \cite{AP}, \cite{HRX} and \cite{L}. We review below
some of the pertinent constructions and results in  those papers.

Our paper is devoted to the external theory within zonotopal
algebra, and solely focuses on the homogeneous, continuous setup (as
\cite{AP,HRX,L} do).  {\it Our goal is to provide a unifying theory
that encompasses all the above-listed approaches and constructions.}
We fix, as above, a linear $X:\R^N\to\R^n$, represent $X$ as an
$n\times N$ matrix (say, with respect to the standard bases in
$\R^N$ and $\R^n$), and treat $X$ also as the multiset of its
columns. Zonotopal algebra, in each of its three layers, continues
with the introduction of a pair of homogeneous polynomial spaces;
the first is usually dubbed a ``$\mathcal{P}$-space", is connected
to the geometry of the zonotope and is explicit. The second is known
as a ``$\mathcal{D}$-space" and, as a rule, is defined implicitly as
the joint kernel of a suitable set of differential operators, whose
corresponding ideal of differential operators is labeled  a
``$\calJ$-ideal''. The ideal $\calJ$ and its corresponding kernel
$\mathcal{D}$ are associated with the geometry of the hyperplane
arrangement.

\bigskip\noindent
{\bf Zonotopal algebra, central.} Let us describe in further detail
the setup.  With the multiset
$X\subset\R^n$ given and fixed, we associate every $x\in X$ (i.e.,
every column of the matrix $X$), with the linear form
$$p_x:\R^n\to \R \ : \  t\mapsto x\cdot t,$$
(with ``$\cdot$" the standard inner product in $\R^n$)
and the corresponding differential operator
$$p_x(D)$$
i.e., the directional derivative $D_x$ in the $x$-direction. Given a
(multi)subset $Z\subset X$, we further denote
$$p_Z:=\prod_{x\in Z}p_x.$$
The central zonotopal algebra setup assumes then that $X$ is of full
rank $n$, and continues with a partition of $2^X$ into the
collection of long subsets
$$L(X):=\{Z\subset X\mid \rank(X\bks Z)<n\},$$
and its complementary collection of short subsets:
$$S(X):=2^X\bks L(X).$$
The central $\mathcal{P}$-space $\mathcal{P}(X)$ is defined with the
aid of the short sets in $S(X)$:
$$\mathcal{P}(X):=\Span \{p_Z: \ Z\in S(X)\}.$$
The long sets generate the $\calJ$-ideal:
$$\calJ(X):=\Ideal\{p_Z
\mid Z\in L(X)\}=\Ideal\{p_Z\mid Z\cap B\neq \emptyset,\, \forall B\in
\B(X)\},$$
with
\begin{equation}\label{defbasis}
\mathbb{B}(X)
\end{equation}
the set of {\bf bases} of $X$, i.e., subsets of
$X$ that form a basis for $\R^n$.
The $\mathcal{D}$-space is then the kernel of $\mathcal{J}(X)$:
$$\mathcal{D}(X):=
\{f\in \Pi\mid p(D)f=0,\,\forall p\in \calJ(X)\}=
\{f\in \Pi\mid p(D)f(0)=0,\,\forall p\in \calJ(X)\},$$
with
$$\Pi=\C[t_1,\ldots,t_n]$$ the space of all polynomials  in $n$ variables.
It is known, \cite{DR}, that
$$\mathcal{P}(X)\oplus \calJ(X)=\Pi,$$
which is equivalent to the statement that the pairing
\begin{equation}
\label{defpairing}
\langle\cdot,\cdot\rangle: \Pi\times \Pi \ : \  (p,q)\mapsto \langle p,q\rangle:=
p(D)q(0)
\end{equation}
induces a linear bijection between $\mathcal{P}(X)$ and
$\mathcal{D}(X)'$, i.e., every linear functional
$\lam\in\mathcal{D}(X)'$ is uniquely represented by some
$p\in\mathcal{P}(X)$: $\lam q=\langle p,q\rangle$,
$q\in\mathcal{D}(X)$. Moreover, it is known, \cite{DM}, \cite{DR}, that
$$\dim\mathcal{P}(X)=\dim\mathcal{D}(X)=\#\B(X).$$
In the sequel we will also need the
(multi)set
\begin{equation}\label{defind}
\mathbb{I}(X)
\end{equation}
of all {\bf independent subsets} of $X$ (i.e., all subsets of the bases).

\bigskip\noindent
{\bf Connection with geometry and the least map.} As said, two geometries
underlie zonotopal algebra. We discuss here the connection of
$\mathcal{D}(X)$ with hyperplane arrangements; cf.\ \cite{BDR} and \cite{HR}
for connections of $\mathcal{P}(X)$ and related spaces to zonotopes. One
starts, \cite{DR}, by associating each $x\in X$ with a constant
$\lam_x\in\R$. Set
\begin{equation}\label{defqx}
q_x:=p_x-\lam_x.
\end{equation}
Each $B\in\B(X)$ defines a {\it vertex} $\b(B)\in\R^n$, viz,
the common zero of the polynomials $(q_x)_{x\in B}$. Assume that the
map
$$\b:\B(X)\to\R^n$$
is injective (which is the generic case in terms of the selection of
$(\lam_x)$), i.e., no point $v\in\R^n$ is a common zero for $n+1$
polynomials $q_x$, $x\in X$. The set $\b(\B(X))$ is then the vertex
set of the hyperplane arrangement $\calH(X)$ generated by the zero sets
$H_x$ of $q_x$, $x\in X$.
\begin{example}

In case $$X=\left(
           \begin{array}{cccc}
             0 & 1 & 1 & 1 \\
             1 & 0 & -1 & 1 \\
           \end{array}
         \right)=:(x_1,x_2,x_3,x_4)
$$the hyperplane arrangement $\calH(X)$ is as follows:
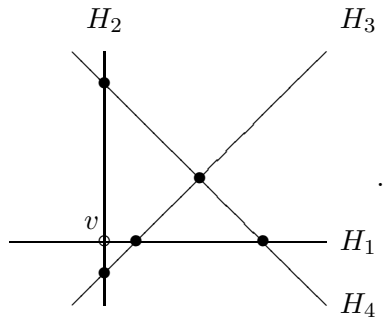
\begin{center}
$ \xy0;/r.2pc/: (20,-10)*{}="G"; (-30,-10)*{}="H"; (-15,20)*{}="L";
(-15,-20)*{}="M"; (-15,25)*{H_2}; (25,-10)*{H_1};
 "G"; "H"**\dir{-};
 "L"; "M"**\dir{-};
 (-20,-20)*{}="y3a"; (20,20)*{}="y3b";(25,25)*{H_3};
 "y3a"; "y3b"**\dir{-};
 (20,-20)*{}="y1a"; (-20,20)*{}="y1b";(25,-20)*{H_4};
 "y1a"; "y1b"**\dir{-};
 (0,0)*{\bullet};(-15,15)*{\bullet};(-15,-15)*{\bullet};
 (-10,-10)*{\bullet};(10,-10)*{\bullet};(-15,-10)*{\circ};
 (-17,-7)*{v}
\endxy
.$
\end{center}
Here,  $H_i$ is the zero set of $q_{x_i}$, and the chosen constants
are $(0,0,1,5)$.
There are six vertices in  $\b(\B(X))$. For example, the marked vertex
$v$ is
$\b(\{x_1,x_2\}).$

\end{example}
 We apply then to the vertex set $\b(\B(X))$ the {\it least map}
of \cite{BR90}. The least map associates each finite $\Tet\subset
\R^n$ with a polynomial space $\Pi(\Tet)$ in the following manner.
One defines first an exponential space
$$\Exp(\Tet):=\Span\{e_\tet\mid \tet\in\Tet\},\quad e_\tet: t\mapsto
e^{\tet\cdot t}.$$
Each $f\in\Exp(\Tet)$ is an $n$-variate entire function hence admits an
expansion
$$f=f_0+f_1+\ldots,$$
with $f_j$ a homogeneous polynomial of degree $j$. Define
$$f\least:=f_j, \quad j:=\max\{j'\ge 0\mid\ f_m=0,\ \all m<j'\}.$$
Then:

\begin{theorem} [\cite{BR90}] \label{leastthm}With
$$\Pi(\Tet):=\Span\{f\least\mid\ f\in\Exp(\Tet)\},$$
the restriction map from $\R^n$ to $\Tet$: $p\mapsto p_{|_\Tet}$ is
a bijection between $\Pi(\Tet)$ and $\R^ {\Tet}$. In particular,
$$\dim\Pi(\Tet)=\#\Tet.$$
\eop
\end{theorem}

\noindent
The {\bf least map} is then the association
$$\Tet\mapsto \Pi(\Tet),$$
with the polynomial space $\Pi(\Tet)$ known as the {\bf least space}
(of $\Tet$).

\smallskip
Applying the least map to the vertex set $\b(\B(X))$ of the
hyperplane arrangement $\calH(X)$ one obtains the following
algebro-geometric interpretation of the equality
$\dim\calD(X)=\#\B(X)$:

\goodbreak
\begin{theorem} [\cite{BR91}] $\Pi(\b(\B(X)))=\mathcal{D}(X)$, for every
generic choice of the constants $(\lam_x)_{x\in X}$.
\eop
\end{theorem}

The duality between
$\mathcal{P}(X)$ and $\mathcal{D}(X)$ then implies
that $\mathcal{P}(X)$
interpolates correctly on $\b(\B(X))$; we explain and elaborate on this
point in the sequel. In any event, the connection between the
explicit $\calP(X)$ and the explicit $\b(\B(X))$ can be established
directly without a recourse to the implicit $\calD(X)$,
and is done as follows. We use here the notation
$$q_Z:=\prod_{x\in Z}q_x,\quad Z\subset X.$$

\begin{theorem}
[\cite{DR}]%
\footnote{The result as stated is straightforward once one knows that
$\dim\mathcal{P}(X)=\#\B(X)$. However, the construction of the Lagrange basis
was originally used in \cite{DR} to prove this dimension formula.}
 \ Assuming the selection of constants
$(\lam_x)$ above to be generic, the polynomials $(q_{X{\bks}
B})_{B\in\B(X)}$ form a Lagrange basis for $\mathcal{P}(X)$
with respect to the vertex set $\b(\B(X))$: given $B\in
\B(X)$, the polynomial $q_{X{\bks} B}$ vanishes at all points
of $\b(\B(X))$ other than $\b(B)$.
\end{theorem}

\bigskip\noindent {\bf External zonotopal algebra.} External zonotopal
algebra (in its homogeneous continuous setup) deals with polynomial
spaces that extend the (central) $\mathcal{P}$- and $\mathcal{D}$-
spaces above. This is done, \cite{HR}, by introducing a
complementary set $Y\subset\R^n$ and ordering the elements of $Y$ in
some fixed way
\begin{equation}\label{defY}
Y=\{y_1,y_2,\ldots\}.
\end{equation}
In \cite{HR} and \cite{HRX} $Y$ is a fixed, arbitrary, ordered basis for
$\R^n$. In the present paper, $Y$ is a (sufficiently long, see below)
sequence of
vectors in general position in $X\cup Y$: no vector $y\in Y$ is in the span
of fewer than $n$ vectors in $(X\cup Y)\bks y$. We assume $X\cup Y$ to
have full rank $n$, but make no such assumption on $X$.\footnote{To be sure,
a basis $B$ is for $\R^n$; therefore, if $\rank X<n$, we have
$\B(X)=\emptyset$, ignoring the fact that $X$ has an intrinsic rank and hence
a possibly non-empty set of intrinsic bases.}

Whatever the choice of the complementary (ordered) matroid $Y$ is,
one continues by selecting suitably a subset $\B'$ from the
basis set of the matroid $X\cup Y$:
$$\B'\subset \B(X\cup Y).$$
The selection is {\bf external} whenever $\B(X)\subset \B'$. The
corresponding $\calJ$-ideal (which is well-defined regardless whether
$\B'$ is external or not) is then defined as
\begin{equation}\label{J_B'}
 \calJ_{\B'}:=\Ideal\{p_Z\mid Z\subset X\cup Y,\ Z\cap B\not=\emptyset,\, \all
B\in \B'\}.
\end{equation}
The corresponding $\calD$-space
$$\calD_{\B'}$$
is then defined as the {\it kernel} of $\calJ_{\B'}$, i.e., the space
of all polynomials that are annihilated by all the differential
operators induced by $\calJ_{\B'}$. Equivalently, $\calD_{\B'}$
in the annihilator of $\calJ_{\B'}$ with respect to our pairing
(\ref{defpairing}):
$$f\in \calD_{\B'}\iff \inpro{f,\calJ_{\B'}}=0.$$
While we are interested in particular, structured, choices of $\B'$, we have
the following unqualified estimate on $\dim\calD_{\B'}$:

\begin{theorem} [\cite{BR91}]\label{br91thm} For an arbitrary $\B'\subset \B(X\cup Y)$,
\begin{equation}\label{BR91ineq}
\dim\calD_{\B'}\ge \#\B'.
\end{equation}
\end{theorem}
Note that in the central case, when $\B'=\B(X)$, there is an equality
in (\ref{BR91ineq}). Indeed, we are  only interested in this particular case:

\begin{definition} We say that the external selection
$\B(X)\subset\B'\subset \B(X\cup Y)$ is {\bf coherent} if
$$\dim\calD_{\B'}=\#\B'.$$
\end{definition}
As said, \cite{HR} was the first to consider an external setup. It
chose $Y$ above to be an arbitrary (ordered) basis for $\R^n$, and
defined a set injection
$$\ex: \I(X)\to \B(X\cup Y),$$
via a greedy extension of each independent set to a basis using the elements
of $Y$. The corresponding $\calD$-space is then denoted there as
$\calD_+(X)$ and its corresponding ideal $\calJ_+(X)$. It is indeed proved
in \cite{HR} that $\B':=\ex(\I(X))$ is
coherent:
$$\dim\calD_+(X)=\#\I(X).$$
Subsequently the reference \cite{HRX} generalized the above external setup
by restricting the extension map $\ex$ to a subset $\I'$ of $\I(X)$ that
satisfies  an additional assumption:

\begin{definition} With $X$ as above, let $\I'\subset \I(X)$. We say that
$\I'$ is {\bf solid} if, given any $I'\in \I'$ and $I\in \I(X)$,
$$\Span I'\subset \Span I\implies I\in \I'.$$
\end{definition}
\cite{HRX} proved that $\B':=\ex(\I')$ is coherent, too, provided that
$\I'$ is solid (in $\I(X)$).

Both references \cite{HR} and \cite{HRX} build also suitable
hyperplane arrangements, select a subset $V$ of the vertex set of
the arrangement and prove that their corresponding $\calD$-space is
the least space of the vertex set $V$. We refer to \cite{HR,HRX} for
details.

\bigskip\noindent
{\bf $\calP$-spaces.} The original external version $\calP_+(X)$ was
introduced independently in \cite{PSS} and \cite{HR}. It is
defined as
$$\calP_+(X):=\Span\{p_Z\mid\ Z\subset X\}.$$
It is proved in \cite{HR} that $\calP_+(X)$ and $\calD_+(X)$ are
dual\footnote{Note that $\calP_+(X)$ depends only on $X$, while $\calD_+(X)$
depends on the order basis $Y$, too. The duality is thus valid regardless of
the way we choose $Y$.} or, in other words, that
$$\calJ_+(X)\oplus \calP_+(X)=\Pi.$$
This property definitely implies that $\dim\calP_+(X)=\dim\calD_+(X)$,
hence
$$\dim\calP_+(X)=\#\I(X).$$
In \cite{AP}, a more general version is defined: one fixes $k\ge 0$,
denotes by
$$\Pi_k$$
the space of all polynomials of degree $\le k$ (in $n$ variables), and
defines
$$\calP_{+k}(X):=\sum_{Z\subset X}p_Z\Pi_k.$$
The following can be deduced from \cite{AP}:
\begin{theorem}
$$\dim \calP_{+k}(X)=\sum_{I\in \I(X)}{n+k-\#I\choose k}.$$
\end{theorem}
The original external space $\calP_+(X)$ thus corresponds to the
case $k=0$.

Two other papers introduce and study external $\calP$-variants:
\cite{HRX}, given a solid $\I'\subset \I(X)$, defines an intermediate
$$\calP(X)\subset \calP_{\I'}\subset \calP_+(X)$$
and proves its duality with $\calD_{\I'}:=\calD_{\ex{\I'}}$.
Recently, Lenz, in \cite{L}, introduced a setup that generalizes
\cite{HRX} as well as \cite{AP}: given a nonnegative integer $k$ and
an upper set $J\subset \mathcal{L}(X)$, where $ \mathcal{L}(X)$ is
the lattice of flats of the matroid $X$, he defines
$$\calP_{+k,J}:=\sum_{Z\subset X}p_{Z}\Pi_{k+\esp(X\bks Z)},$$
with $\esp$ the indicator function of $J$, and, for
$Z\subset X$, $\esp(Z):=\esp(\Span Z)$. He proved a suitable
dimension formula for this $\calP$-space.

\bigskip\noindent
{\bf Homogeneous basis for $\calP(X)$ and Hilbert functions.}  There
are no known explicit constructions of bases for $\calD$-type
spaces. In contrast, there are such basis constructions for the
central $\calP(X)$ and each of the external variants discussed
above. These constructions allow one (in theory) to compute the
Hilbert functions of those $\calP$-spaces. The only ``real''
construction is the one that was given in \cite{DR} for the central
$\calP(X)$ and is done as follows. Given $X$ as above, one fixes an
arbitrary order $\prec$ on the elements of $X$. Then, given
$B\in\B(X)$, one defines
\begin{equation}\label{defxb}
X(B):=\{x\in X\bks B\mid x\not\in\ {\rm span}\{b\in B\mid b\prec x\}\}.
\end{equation}
The cardinality of $X(B)$ is  intimately connected to the {\it
external activity} of $B$, which equals to $ \#(X\bks B)-\#X(B)$ (see, e.g.,
\cite{Biggs}).

\begin{theorem}\label{homobasis} {\rm \cite{DR}}
The polynomials
$$p_{X(B)},\quad B\in \B(X)$$
form a basis for $\calP(X)$.
\end{theorem}

\medskip
The construction of homogeneous bases for external $\calP$-spaces is
obtained as a variation of the above construction, using the
following approach.  Suppose that we have defined a $\calD$-space
$\calD_{\B'}$, corresponding to the basis set $\B'\subset \B(X\cup
Y)$, and a related $\calP_{\B'}$ and proved a duality between the
$\calD$- and the $\calP$- space. Now, necessarily,
$$\calP_{\B'}\subset \calP(X\cup Y).$$
Thus,  we construct a homogeneous basis for $\calP(X\cup Y)$ as
above, and select the basis polynomials that correspond
to $B\in \B'$. These polynomials are automatically linearly independent.
Assuming that $\B'$ is coherent,  we combine this coherence together
with the assumed duality between $\calP_{\B'}$ and $\calD_{\B'}$
to conclude that
$$\dim\calP_{\B'}=\dim\calD_{\B'}=\#\B'.$$
Thus, the polynomials selected above will form a basis for
$\calP_{\B'}$ once we show that each of them actually lies in
$\calP_{\B'}$.

This approach was, at least implicitly, used in \cite{HR,HRX} for the
construction of homogeneous bases for the external $\calP$-spaces that
were studied there. \cite{AP,L} used other methods since they introduced
$\calP$-spaces without corresponding $\calD$-spaces.

Given any homogeneous polynomial space, $\calP$, the construction of
a homogeneous basis $(Q_B)_{B\in \B'}$, with $\B'$ some index set,
allows one to compute the Hilbert function of that $\calP$-space,
i.e., the function
$$h_X:k\mapsto \dim(\calP\cap \Pi_k^0),$$
with $\Pi_k^0$ the space of homogeneous polynomials of degree $k$. In the
description above, the Hilbert function is combinatorial/matroidal:
\begin{equation}\label{defhx}
h_X(k)=\#\{B\in \B'\mid \#(X(B))=k\}.
\end{equation}

\bigskip\noindent
{\bf Our setup.} Our setup provides a general unified theory and analysis
that captures all above-mentioned efforts as special cases. A key to our
approach is the simultaneous development of the two
types of spaces: $\calD$- and $\calP$- ones. Given
our multiset $X$ (which, in contrast with previous studies like
the one in \cite{HR}, is not assumed to be necessarily of full rank),
we begin with an assignment
$$\strangek: 2^X\to\N,$$
which is solid:

\begin{definition}\label{solid} An assignment $\strangek$ as above is {\bf solid} if,
given $Z,Z'\subset X$, we have
$$\Span Z\subset \Span Z' \implies \strangek(Z)\le \strangek(Z').$$
\end{definition}

Given a solid assignment $\strangek$, we define the $\calP$-space as
$$\calP_{\strangek}:=\sum_{Z\subset X}p_{X{\bks} Z}\Pi_{\strangek(Z)}.$$
In order to augment this definition with a corresponding $\calD$-space, we
choose $Y=\{y_1,y_2,\ldots\}$  to contain sufficiently many vectors in general
position (in $X\cup Y$, cf.\ the discussion after (\ref{defY})), and denote
\begin{equation}\label{defYm}
Y_{i}:=\{y_1,\ldots,y_{i}\}, \quad i>0,
\end{equation}
and $Y_i=\emptyset$ if $i\le 0$.
The associated basis set $\B':=\B_{\strangek}\subset \B(X\cup Y)$ is defined as
follows:
\begin{equation}\label{B_k}
\B_{\strangek}:=\{B\in \B(X\cup Y)\mid\ B\cap Y\subset Y_{m(B\cap
X)}\},
\end{equation}
where, for an independent $I\in \I(X)$,
$$m(I):=\strangek(I)+n-\#I.$$
It follows that each independent $I\subset X$ can be extended in
${m(I)\choose \sk(I)}$ different ways to a basis in $\B_\sk$, hence
that
$$\#\B_{\strangek}=\sum_{I\in \I(X)}{m(I)\choose \strangek(I)}.$$

\begin{example}\label{eg1}Let $X=\{x_1,x_2\}\subset \mathbb{R}^2$, where
$x_1=\binom{1}{0}$ and $x_2=\binom{0}{1}$. Assume that $\sk$ is
solid and that $\strangek(x_1)=\strangek(x_2).$ It then easily
follows that
\begin{equation}\label{bk}
\#\bprime=\binom{2+\sk(\emptyset)}{2}+2\sk(x_1)+3.
\end{equation}
\end{example}

\medskip
The $\calD$-space $\calD_\sk$ is defined as
$$\calD_\sk:=\calD_{\B_\sk}=\ker\calJ_{\B_\sk},$$
where $\calJ_{\B_\sk}$ is defined in (\ref{J_B'}) with respect to the choice
$\B'=\B_{\sk}$. As before, we associate each $z\in X\cup Y$ with a
constant $\lam_z$ and assume the assignment to be generic. Every
$B\in\B(X\cup Y)$ then corresponds to $\b(B):=$ the common zero of
the polynomials $(q_z)_{z\in B}$, and, by assumption, the map
$$\b:\B(X\cup Y)\to \R^n \quad : \quad B\mapsto \b(B)$$
is injective. We denote
$$V_{\strangek}:=\b(\B_{\strangek}).$$

At this generality, we are able to prove only partial results:

\begin{theorem}\label{thmpartial} Let $\strangek$ be a solid assignment. Then:
\begin{itemize}
\item  $\B_{\strangek}$ is coherent, i.e.,
$\dim\calD_{\strangek}=\#\B_{\strangek}$. Furthermore,
$\Pi(V_{\strangek})=\calD_{\strangek}$.
\item $\calP_{\strangek}+\calJ_{\strangek}=\Pi.$
\item $\calP_{\strangek}$ contains a Lagrange basis for $V_{\strangek}$:
for each $v\in V_{\strangek}$ there exists $L_v\in \calP_{\strangek}$,
such that $L_v$ vanishes on $V_{\strangek}\bks v$, but not at $v$.
\end{itemize}
\end{theorem}

\noindent
A few remarks are then in order:

1. We provide an explicit construction of the aforementioned Lagrange
basis.

2. The second result in the above theorem implies that
$\dim\calP_{\strangek}\ge \dim\calD_{\strangek}$. Simple examples show that
this inequality can be sharp.

3. The third result implies that $\dim\calP_{\strangek}\ge
\#V_{\strangek}=\#\B_{\strangek}$. This inequality follows also from
the second result, since (\ref{BR91ineq}),
$$\dim\calD_{\strangek}\ge \#\B_{\strangek}$$
even without the solid assumption.

4. We also identify in $\calP_\sk$ a family of $\#\B_\sk$ linearly
independent homogeneous polynomials. That construction not only reproves
the inequality $\dim\calP_\sk\ge \#\B_\sk$, but also provides a lower bound
on the values assumed by the Hilbert function of $\calP_\sk$.

\medskip
Stronger results are obtained once
we make an additional assumption:

\begin{definition} We say that an assignment $\sk$ is {\bf incremental}
if, for every $Z\subset X$ and $x\in X$,
$$\strangek(Z\cup x)\le \strangek(Z)+1.$$
\end{definition}

Indeed,
we obtain a complete theory for assignments that are both solid and
incremental:

\begin{theorem}\label{mainthmintro} Assume the assignment $\strangek$ to be
solid and incremental.  Set $X':=X\cup Y$, and, for $I\in \I(X)$,
$X'_I:=X\cup Y_{m(I)}$.
Then
\begin{itemize}
\item The polynomials
$$q_{(X'_{B\cap X})\bks B},\quad B\in \B_{\strangek}$$
form an inhomogeneous basis for $\calP_{\strangek}$.
In particular,
$$\dim\calP_{\strangek}=\#\B_{\strangek}.$$
\item The polynomials
$$p_{X'(B)},
\quad B\in \B_{\strangek}$$
form a homogeneous basis for $\calP_{\strangek}$.%
\footnote{ The notation
$X'(B)$ is defined in (\ref{defxb}), with $X$ there replaced by $X'$ here.}
\end{itemize}
\end{theorem}

It follows from this result that the Lagrange basis in Theorem
\ref{thmpartial} is also
a basis for $\calP_{\strangek}$. Also, we can now conclude that
$$\calJ_{\strangek}\oplus \calP_{\strangek}=\Pi,$$
or in other words that $\calP_{\strangek}$ and $\calD_{\strangek}$ are dual
to each other.

Finally, the construction of a homogeneous basis for $\calP_\sk$ 
leads to a combinatorial formula for the Hilbert function $h_\sk$ of 
$\calP_\sk$, which, due to the duality between $\calP_\sk$ and $\calD_\sk$, 
is also the Hilbert function of $\calD_\sk$: for $j\ge 0$ we have
\begin{equation} \label{hilform}
h_\sk(j)=h_X(j)+\sum{j-\#X(I)+n-\#I-1\choose n-\#I-1},
\end{equation}
where the sum runs over all $I\in \I(X)\bks \B(X)$ for which
$j-\sk(I)\le\#(X(I))\le j$, and with $h_X$ the Hilbert function of
$\calP(X)$ (cf.\ (\ref{defhx})).

\begin{proof} Given $j\ge 0$, we need to count the number of polynomials
$p_{X'(B)}$ in the homogeneous basis for $\calP_\sk$ that are of degree $j$
(cf.\ Theorem \ref{mainthmintro}). In other words, we need to find out the
number
$$\#\{B\in \B_\sk\mid \#X'(B)=j\}.$$
Since
$$h_X(j)=\#\{B\in \B(X)\mid \#X'(B)=j\},$$
we need only to focus on $\B_\sk\bks\B(X)$.
To this end, we write $B=I\cup J\in\B_{\sk}\bks \B(X)$ with $I\subset X$
and $J\subset Y$; also, let $y_k$ be the maximal element of $J$. Then
$X'(B)=X(I)\cup (Y_k\bks J)$. Since  we need to have
$\#X'(B)=j$, it is necessary that
$0\le j-\#X(I)\le \sk(I)$. Once our $I$ is fixed, $y_k$, the last element
of $J$, has to satisfy that $k=j-\#X(I)+n-\#I$.
Then, we can freely choose
the remaining $n-\#I-1$ elements from $Y_{k-1}$. This validates the given
formula.
\end{proof}
\begin{example}Consider $X=\{x_1,x_2,x_3\}$, where
$x_1=\binom{1}{0}$ and $x_2=x_3=\binom{0}{1}$. Then, we have
$X(\emptyset)=\{x_1,x_2,x_3\}$, $X(\{x_1\})=\{x_2,x_3\}$,
$X(\{x_2\})=\{x_1\}$, $X(\{x_3\})=\{x_1,x_2\}$,
$X(\{x_1,x_3\})=\{x_2\}$ and $X(\{x_1,x_2\})=\emptyset$. Assume
$\sk(\emptyset)=\sk(\{x_1\})=1$ and $\sk(\{x_2\})=\sk(\{x_3\})=2$.
Then for $j=4$, the independent sets in the sum (\ref{hilform}) are
$\emptyset$ and
$\{x_3\}$, and we have $h_{\sk}(4)=3$. For $j<4$, one finds out that
$h_\sk(j)=j+1$, hence that $\Pi_3\subset \calP_\sk$. Note that
$\dim\calP_\sk=\#\B_\sk=13$ here.
\end{example}

\section{Construction and analysis of $\mathcal{D}_\strangek$}

The main objective in this section is to show that the space
$\calD_\strangek$ is coherent, whenever $\strangek$ is solid.
Thus, the main result in this section is the following:
\begin{theorem}\label{dimd}
$\B_k$ is coherent for all solid assignments $\sk$:
$$\dim\mathcal{D}_{\strangek}=\# \bprime.$$
\end{theorem}

 Recall that the
lower bound $\dim\mathcal{D}_{\strangek}\ge \# \bprime$ is valid,
Theorem \ref{br91thm}, without any conditions or assumptions on
$\B_\sk$. The solid assumption on $\sk$, thus, leads to a matching
 upper
bound. In proving this matching bound, we will invoke the notion of {\it
placability}:
\begin{definition}Let $X$ be a matroid and $\emptyset \neq
\mathbb{B}'\subset\mathbb{B}(X)$.
\begin{enumerate}
\item Given $x\in X$, the actions of {\bf deletion} of $x$ and {\bf
restriction} to $x$ decompose $\B'$ into
$$
\mathbb{B}'_{\backslash x}:=\{B\in \B'\mid x\not\in B\},\quad
\hbox{and}\quad \mathbb{B}'_{/x}
:=\{B\in \B'\mid x\in B\}.$$
\item An element $x\in X$ is {\bf placable} in
$\mathbb{B}'$ if for each $B\in \mathbb{B}'$, there exists an
element $a\in B$ such that $ (\{x\}\cup B)\bks \{a\}\in \mathbb{B}'$.
\item  A (placable) {\bf split} of $\mathbb{B}'$ is a set
partition $\mathbb{B}'_{/x}\sqcup \mathbb{B}'_{\backslash x}$ by a placable
element $x$ such that both $\mathbb{B}'_{/x}, \mathbb{B}'_{\backslash x}\neq
\emptyset$ .
\item  We say that $\mathbb{B}'$ is {\bf placible}
if one of the following two conditions holds:
\begin{enumerate}
\item $\mathbb{B}'$ is a singleton.
\item There exists $x\in X$ which is placable in $\mathbb{B}'$, for which
$\mathbb{B}'_{/x}$ and $\mathbb{B}'_{\backslash x}$ are, each, non-empty
and placible.
\end{enumerate}
\end{enumerate}
\end{definition}

Note that Part 4 of the above definition is inductive; this inductive
definition is valid, since we assume
both $\mathbb{B}'_{/x}$ and $\mathbb{B}'_{\backslash x}$ to be
nonempty.

\smallskip
The following is known:
\begin{lemma}[\cite{BRS96}]\label{lemplace} Let
$\mathbb{B}'\subset\mathbb{B}(X)$.  If $\mathbb{B}'$ is placible, then
$\dim\mathcal{D}_{\mathbb{B}'}\le\#\mathbb{B}'$.
\end{lemma}

Thus, in view of the above lemma, the inequality
$\dim\mathcal{D}_{\strangek}\le \# \bprime$ will follow once we show
that $\bprime$ is placible, as we do now.

First, recall from (\ref{B_k}) that
$$\B_{\strangek}:=\{B\in \B(X\cup Y)\mid\ B\cap Y\subset Y_{m(B\cap X)}\}.$$
Given two disjoint subsets, $A,C$, of $X$, we denote
$$\bprime_{,A,C}:=\{B\in \B_{\strangek}\mid B\cap A=
\emptyset, C\subset B\}.$$ Notice that it is possible that another
pair $A',C'$ defines the same set: $\bprime_{,A',C'}=\bprime_{,A,C}$.
Assume in the following proposition that $A$ and $C$ are maximal. It
then follows, since $C\subset B$, for each $B\in\bprime_{,A,C}$,
that $\Span C\subset A\cup C$.

\begin{proposition}\label{placable}Assume that $\sk$ is solid.
Then each element $x\in X\backslash(A\cup C)$ is placable in
$\bprime_{,A,C}$.
\end{proposition}
\begin{proof}Let $x\in X\backslash(A\cup C)$ and $B\in\bprime_{,A,C}$.
We need to show that we can replace some element of $B$ by $x$ to
obtain another basis in $\B_\sk$.  This is trivial if $x\in B$. So
we assume that $B$ contains $C$ and is disjoint of $A$, $x$, and (due
to the maximality of $A$) $\Span C\bks C$. Denote
$I:=B\cap X$ and $J:=B\cap Y$. There are two cases to consider:
\begin{enumerate}
\item  $x\notin\Span (I)$. In this case, we
replace the last element $y$ of $J$
with $x$. We claim that
$$B':=(I\cup\{x\})\cup (J\backslash \{y\}):=I'\cup J'\in
\bprime_{, A,C}.$$ First, it is clear that $I'\cap A=\emptyset$, since
$I\cap A=\emptyset$, and $x\not\in A$.
Also, $C\subset I'$, since $C\subset I$.
Therefore, we only need to show that $J'\subset Y_{m(I')}$. We know,
by assumption, that $J\subset Y_{m(I)}$. Since $y$ is the {\it last}
element of $J$, we conclude that $J' \subset Y_{m(I)-1}.$ However,
$$m(I)-1=\sk (I)+\co I-1=\sk (I)+\co I'\le
\sk (I')+\co I'=m(I'),$$
with the inequality following from the solid property of $\sk$. Consequently,
$J'\subset Y_{m(I')}$ as required.

\item  $x\in\Span  I$. Since we assume that $x\notin A$, and $A$ is maximal, we have
$x\notin \Span C$. So there exists $a\in I\bks C$ such that, with
$I':=\{x\}\cup I\bks\{a\}$, $\Span I=\Span I'$.
We now claim that $$B':=I'
\cup J\in \bprime.$$ Here all the requisite conditions are immediate.
First,
$I'\cap A=\emptyset$, since $I\cap A=\emptyset$, and $x\notin A$.
Second, $C\subset I'$ since $a\notin C$ and $C\subset I$.  Finally,
since $\Span I=\Span I'$, and since $\sk$ is solid, we must have
$\sk (I')=\sk (I)$. Since also  $\# I'=\# I$, we conclude that
$m(I)=m(I')$. Therefore, the required inclusion $J\subset Y_{m(I')}$ follows
from the assumed inclusion
$J\subset Y_{m(I)}.$
\end{enumerate}
\end{proof}

Now, we can build a binary tree whose root is $\B_\sk$, and with
each branching of a node done by deletion/restriction using some
element $x\in X$. Obviously, every node in such tree is of the form
$\bprime_{,A,C}$. Let us assume that the branching of the node
$\bprime_{,A,C}$ is done by an element $x\in X\bks (A\cup C)$. Such
element was just proved to be placable in $\bprime_{,A,C}$. The maximality
assumption on $A,C$ easily leads to the conclusion that
the split is non-trivial.
We can continue branching the nodes of the tree as much
as it is possible.
Obviously, we will
have to stop only when $X\subset A\cup C$, i.e., $X\bks A\subset C$.
Since we assume $\bprime_{,A,C}\not=\emptyset$, it must be the case
that
 $C=I$ is some independent set and $A=X\backslash I$. So this
node corresponds to the set $\ex(I)$, i.e., bases in $\B_\sk$ which
extend $I$ using elements of $Y_{m(I)}$. If $I\in \B(X)$, we are
done since the node is a singleton, and the same applies if
$\sk(I)=0$. Otherwise, every $y\in Y_{m(I)}$ is placable in every
subset of $\ex(I)$, as one easily verifies.  Thus, we can split
$\ex(I)$ successively using elements of $Y_{m(I)}$ until $\ex(I)$ is
completely split to singletons.

Thus, we have shown that $\B_\sk$ is  placible. {\it Consequently, we can invoke
Lemma \ref{lemplace} to obtain Theorem \ref{dimd}.}

%

\medskip
Next, we return our attention to the inhomogeneous polynomials
$q_z$, $z\in X\cup Y$, the associated hyperplane arrangement
$\calH(X\cup Y)$, and the bijection $\b$ from $\B(X\cup Y)$ onto
the vertex set $V_\sk$ of $\calH(X\cup Y)$ (cf.\ the discussion around
(\ref{defqx})).
Let $\Pi(V_\sk)$ be the least space of $V_\sk$ (cf.\ Theorem
\ref{leastthm}). Now,
$$\dim\Pi(V_\sk)=\#V_\sk=\#\B_\sk=\dim\calD_\sk,$$
with the last equality implied by Theorem  \ref{dimd}.
However, \cite{BR90}, we (always, i.e., even in the absence of the
solid property of $\sk$) have that
$$\Pi(V_\sk)\subset \calD_\sk.$$
Therefore:

\begin{corollary}$\Pi(V_{\strangek})=\mathcal{D}_{\strangek},$ where
$V_{\strangek}=\b(B_{\sk})$.
\end{corollary}

\begin{example}[continuation of Example \ref{eg1}]\label{eg2}
Let $X=\{x_1,x_2\}\subset \mathbb{R}^2$, where $x_1=\binom{1}{0}$
and $x_2=\binom{0}{1}$.
Choose $\strangek(\emptyset)=1$, $\strangek(x_1)=\strangek(x_2)=1$, and
$\strangek(\{x_1,x_2\})=2$. It is trivial to check that
this $\sk$ is solid. We
want to find $\B_{\sk}$ and $V_{\strangek}=\b(B_{\strangek})$ in this
example.

Recall that $m(I)=\strangek(I)+2-\#I$. So $m(\emptyset)=3$ and
$m(x_1)=m(x_2)=m(\{x_1,x_2\})=2$. Therefore it suffices for $Y$ to
have 3 elements: $Y=\{y_1,y_2,y_3\}$. By the definition of $\B_{\sk}$
in (\ref{B_k}), we have
$$\B_{\sk}=\ex(\emptyset)\cup\ex(\{x_1\})\cup\ex(\{x_2\})\cup\ex(\{x_1,x_2\}),$$
with
\begin{align*}
&\ex(\emptyset)=\{\{y_1,y_2\},\{y_1,y_3\},\{y_2,y_3\}\},\\
&\ex(\{x_i\})=\{\{x_i,y_1\},\{x_i,y_2\}\},\,i=1,2,\\
&\ex(\{x_1,x_2\})=\{\{x_1,x_2\}\}.
\end{align*}%
In particular, $\#\B_{\sk}=8$ which matches (\ref{bk}) in Example
\ref{eg1}.

 The associated hyperplane arrangement $\calH(X\cup Y)$ is
depicted in the following figure and
$V_{\strangek}=\b(B_{\strangek})$ corresponds to the vertices of the
arrangement that are marked solid
(viz.\ all vertices but the intersections of
$\{x_1,y_3\}$ and $\{x_2,y_3\}$).
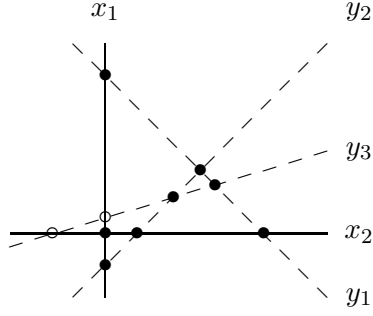
\begin{center} $ \xy 0;/r.2pc/: (20,-10)*{}="G"; (-30,-10)*{}="H";
(-15,20)*{}="L"; (-15,-20)*{}="M"; (-15,25)*{x_1}; (25,-10)*{x_2};
 "G"; "H"**\dir{-};
 "L"; "M"**\dir{-};
 (-20,-20)*{}="y3a"; (20,20)*{}="y3b";(25,25)*{y_2};
 "y3a"; "y3b"**\dir{--};
 (20,3)*{}="y2a"; (-30,-12)*{}="y2b";(25,3)*{y_3};
 "y2a"; "y2b"**\dir{--};
 (20,-20)*{}="y1a"; (-20,20)*{}="y1b";(25,-20)*{y_1};
 "y1a"; "y1b"**\dir{--};
 (0,0)*{\bullet};(-15,15)*{\bullet};(-15,-15)*{\bullet};
 (-10,-10)*{\bullet};(10,-10)*{\bullet};(-15,-10)*{\bullet};
 (-4.29,-4.29)*{\bullet},(2.31,-2.31)*{\bullet};
 (-23.33,-10)*{\circ};(-15,-7.5)*{\circ};
\endxy
$
\end{center}
\end{example}

\section{Construction and analysis of $\mathcal{P}_\sk$}
Recall from the introduction the definition of the polynomials spaces $\Pi$
and $\Pi_k$, $k\ge 0$, and the definition of  $\calP_\sk$:
\begin{equation}\label{P_k}
\mathcal{P}_{\sk} :=\mathcal{P}_{\sk}(X) :=\sum_{Z\subset
X}p_{X{\bks} Z}\Pi_{\sk(Z)}. \end{equation} One of our primary aims
is to establish, under some conditions on $\sk$, a duality between
$\calD_\sk$ and $\calP_\sk$. Thus, we need to have
$$\dim\calP_\sk=\dim\calD_\sk=\#\B_\sk,$$
with the left equality necessary for the duality and the right one our
requirement of coherence.

\begin{example}[Continuation of Example \ref{eg1}]\label{eg3}
Let $X=\{x_1,x_2\}\subset \mathbb{R}^2$, where $x_1=\binom{1}{0}$
and $x_2=\binom{0}{1}$. Let $\strangek(x_1)=\strangek(x_2)=k$,
$\strangek(\emptyset)=j$ with $j\le k$ and
$\strangek(\{x_1,x_2\})=\ell$ with $\ell\ge k$. One can check that
$\strangek$ is solid. As in Example \ref{eg1}, we have
$$\#\B_\sk={j+2\choose 2}+2{k+1\choose 1}+1.$$
In this example, we will compute $\calP_\sk$ explicitly, and compare
its dimension with $\#\B_{\sk}$.

By (\ref{P_k}), we have
$\mathcal{P}_{\strangek}=\Span\{\Pi_kp_{x_1},\Pi_kp_{x_2},\Pi_{\ell},\Pi_jp_
X\}$. There are three cases:
\begin{enumerate}
\item If $\ell>k+1$, we have $\mathcal{P}_{\strangek}=\Pi_{\ell}$
hence $\dim\mathcal{P}_{\strangek}=\binom{\ell+2}{2}$; since we assume
$j\le k< \ell-1$, it is easy to see that we get here
$\dim\calP_\sk>\#\B_\sk$.
\item If $\ell\le k+1$ and $j\le k-1$, we have
$\mathcal{P}_{\strangek}=\Pi_{k+1}$, 
hence $\dim\mathcal{P}_{\strangek}=\binom{k+3}{2}=\binom{k+1}{2}+2
\binom{k+1}{1}+1$; consequently, $\dim\calP_\sk\ge \#\B_\sk$ with
equality  if and only if $j=k-1$.
\item If $\ell\le k+1$ but $j=k$, we have
$$\mathcal{P}_{\strangek}=\Pi_{k+1}+\Span\{p_{x_1}^{m+1}p_{x_2}^{k-m+1}\mid
m=0,\dots,k\}.$$ Therefore,
$$\dim\mathcal{P}_{\strangek}=\binom{k+3}{2}+k+1=\#\B_\sk.$$
\end{enumerate}
Note that the inequality $\dim\calP_\sk\ge \#\B_\sk$ is valid in
each of the above three cases. Our results
in this section make clear that this is not an accident, and is due to the
fact that $\sk$ is solid. At the same time, this example clearly shows that
the solid assumption alone does guarantee our desired {\rm equality}.
To this end, we will revisit the case here in Example \ref{eg4},
and will study closely the situations when equality holds.
\begin{proof}[]
\end{proof}
\end{example}

\medskip
As we just said,
the lower bound on $\dim\calP_\sk$ that was observed in the example
above is true, in general, for every solid assignment $\sk$:
\begin{theorem}\label{pbine} Assume $\sk$ to be solid. Then:
$$\dim \mathcal{P}_{\strangek}\ge \dim \mathcal{D}_{\strangek}=\#\B_\sk.$$
\end{theorem}

The equality $\dim\calD_\sk=\#\B_\sk$ was proved in Theorem
\ref{dimd}. We need thus to prove the inequality assertion. We
provide below three complementary proofs, each revealing a different
property of $\calP_\sk$.

\subsection{First proof of Theorem \ref{pbine}: embedding ${\cal D}_\sk'$ in
${\cal P}_\sk$} The inequality $\dim\calP_\sk\ge \dim\calD_\sk$
follows (directly) from the following stronger result (cf.\ the discussion
above Theorem \ref{thmpartial} for
the definition of the ideal $\calJ_\sk$):
\begin{proposition}\label{sum}
Assume $\sk$ to be solid. Then:
$$\mathcal{J}_{\strangek}+\mathcal{P}_{\strangek}=\Pi.$$
\end{proposition}

\begin{proof}[Proof of Proposition \ref{sum}]
Set
$$A:=\calJ_\sk+\calP_\sk.$$
Let $X'\subset X$, and $Y'\subset Y_{s(X',Y')}$, with
$$s(X',Y'):=n-\rank X'+\#Y'-1.$$
(Note that the definition makes sense even when $\rank X'=n$: we have
then $s(X',Y')<\#Y'$, which merely forces $Y'$ to be empty.)
We claim that, for an
arbitrary polynomial $f$, the product
$$F:=f\,p_{X{\bks} X'}\;p_{Y'}$$
lies in $A$. Choosing $X':=X$ and $Y':=\emptyset$, we will
then obtain the desired result, since $f$ is arbitrary.

In order to prove that $F\in A$, we first fix $X'$ and assume $Y'$
to be ``large enough": $\#Y'>\sk(X')$. We claim that in this case
$F\in \calJ_\sk$, which will follow once we prove that
$$B\cap (Y'\cup X\bks X')\not =\emptyset,$$
for every $B\in \B_\sk$. To this end, we assume that $B\cap X\bks
X'=\emptyset$, and examine $J:=B\cap Y'$. Then $B \bks J\subset X'$,
and since $\sk$ is solid, $\sk(B{\bks} J)\le \sk(X')$. Also, since
$B\in \B_\sk$, $J\subset Y_{m(B\bks J)}$, where
$$m(B\bks J)=\#J+\sk(B\bks J)\le \#J+\sk(X')\le \#
J+\#Y'-1.$$ But we also have $$s(X',Y')\le \# J+\#Y'-1,$$ because
$n-\#J\le \rank X'$. We conclude that $Y'$ as well as $J$ are both
subsets of $Y_{\# J+\#Y'-1}$, implying that these two sets
intersect.

Thus, it remains to show that $F\in A$  when $\#Y'\le \sk(X')$. Note
that the number of pairs $X',Y'$ for which $\#Y'\le \sk(X')$ (and in
addition $X'\subset X$, $Y'\subset s(X',Y')$) is finite. We will
thus prove that $F\in A$ by descending induction on $\#Y'+\#(X\bks
X')$.

 Now, let $X'$ and $Y'$ be as
above. Choose a basis  $I\subset X'$ for $\Span X'$, and let
$J:=Y_{s(X',Y')+1}\bks Y'$. Then $B:=I\cup J$ is a basis for $\R^n$.
Therefore, we can write
$$f=c+\sum_{b\in B}p_b\,f_b,$$
with $c$ some scalar and $(f_b)_{b\in B}$ some polynomials.
Therefore
$$F=c\, p_{X\bks X'}\,p_{Y'}+\sum_{b\in B}p_b\, p_{X\bks X'}\, p_{Y'}\,f_b.$$
Note that $p_{X\bks X'}\,p_{Y'}\in \calP_\sk\subset A$, since
$\#Y'\le \sk(X')$, by assumption. We will use our induction
hypothesis to show that each of the summands
$$p_b\, p_{X\bks X'}\, p_{Y'}\,f_b, \quad b\in B$$
lies in $A$, too. There are two cases to consider:

1. $b\in X'$. In this case, with $X'':=X'\bks b$, we need to check
that $Y'\subset Y_{s(X'',Y')}$, and then the induction will apply.
However, $s(X'',Y')=n-\rank X''+\#Y'-1\ge n-\rank X'+\#Y'-1=s(X',Y')$.
Therefore, $Y'\subset Y_{s(X',Y')}\subset Y_{s(X'',Y')}$.

2. $b\in Y$. In this case, with $Y'':=Y\cup\{b\}$, we need to show
that $Y''\subset Y_{s(X',Y'')}$. However, we have that
$s(X',Y'')=s(X',Y')+1$, and thus
$b\in J\subset Y_{s(X',Y')+1} = Y_{s(X',Y'')}$.
Hence the induction hypothesis applies
here as well.

This completes the inductive step, hence the proof that $F\in A$,
hence the proof of Proposition \ref{sum}, hence the first proof of
Theorem \ref{pbine}.
\end{proof}

\subsection{Second proof of Theorem \ref{pbine}: homogeneous basis}

As noted in the Introduction, we can attempt to construct a
homogeneous basis for a subspace of $\calP_\sk$ by adapting the
basis construction for $\calP(X)$ from \cite{DR}.

Since in our case $\B_\sk\subset \B(X')$, $X':=X\cup Y$, we first
follow \cite{DR} and construct a homogenous basis
$$p_{X'(B)}, \quad B\in \B(X')$$
for $\calP(X')$, as in Theorem \ref{homobasis}. In the actual
construction, we need to order the vectors in $X'$: We choose any
order on $X$, retain the given order on $Y$, and insist that $x\prec
y$ for every $x\in X$ and $y\in Y$.

The polynomials
$$p_{X'(B)}, \quad B\in \B_\sk$$
are trivially linearly independent. Theorem \ref{pbine} will then
follow once we show that each one of them lies in $\calP_\sk$. So,
fix $B\in \B_\sk$. Then, with $I:=X\cap B$, $B\bks I\subset
Y_{m(I)}$. The definition of $X'(B)$ clearly shows that $X'(B)$
contain no vectors that are larger than the maximal vector in $B$.
Therefore, $X'(B)\cap Y\subset  Y_{m(I)}$. Now,
$$\#(X'(B)\cap Y)\le m(I)-\#(B\bks I)=\sk(I).$$
Since $I\subset X\bks X'(B)=:Z$ and $\sk$ is solid, we have that
$$\sk(I)\le \sk(Z),$$
and hence
$$p_{X'(B)}=p_{X'(B)\cap X}\,p_{X'(B)\cap Y}\in p_{X'(B)\cap X}\,\Pi_{\sk(I)}\subset
p_{X\bks Z}\,\Pi_{\sk(Z)}\subset \calP_\sk.$$

Thus, the linearly independent polynomials
$$p_{X'(B)},\quad B\in \B_\sk$$ lie
in $\calP_\sk$, and Theorem \ref{pbine} follows:
$$\dim \calP_\sk\ge \#\B_\sk.$$
\endproof

\subsection{Third proof of Theorem \ref{pbine}: Lagrange basis}

We retain our assumption that $\sk$ is solid, and recall the
definition of the inhomogeneous polynomials $q_x$, $x\in X\cup Y$,
together with the assignment
$$\b:\B_\sk\to\R^n$$
that assigns to each basis the common zero of the polynomials
$q_x$, $x\in B$. Also, $V_\sk:=\b(\B_\sk)$. We will show that
$\calP_\sk$ contains a Lagrange basis with respect to $V_\sk$:

\begin{proposition}\label{conlang} Assume that $\sk$ is solid and let
$V_\sk$ be as above. For every $B\in \B_\sk$, there exists
$L_B\in\calP_\sk$ such that $L_B(\b(B))\not=0$, while $L_B$ vanishes
on $V_\sk\bks \b(B)$.
\end{proposition}

Obviously, the above Lagrange polynomials are linearly independent,
and therefore Proposition \ref{conlang} implies that
$$\dim\calP_\sk\ge \#V_\sk=\#\B_\sk,$$
providing thereby another proof to Theorem \ref{pbine}.

%
%
%
%
%
\medskip
Before we embark on the proof of the Proposition, we mention the following
simple fact:

\begin{lemma}\label{pq} Assume that $\sk$ is solid, 
let $Z\subset X$ and let $f$ be a polynomial of degree
no more than $\sk(X\bks Z)$. Then $q_Z\,f\in \calP_{\sk}$.
\end{lemma}
\begin{proof}Expanding $q_Z$, we have that $q_Z\,f$ is a linear
combination of $p_{Z'}\,f$ for some $Z'\subset Z$. Since $\sk$ is solid,
$\deg(f)\le \sk(X\bks Z)\le \sk(X\bks Z')$, so we have that $p_{Z'}\,f\in
\calP_\sk$. Therefore, $q_Z\,f\in \calP_{\sk}$
\end{proof}

Our next task is to construct the aforementioned Lagrange basis.
So, we fix $B\in \B_{\strangek}$, and denote by $v\in V_{\strangek}$ the
corresponding vertex $v=\b(B)$. Given $x\in X\cup Y$, we have that
$q_x(v)=0$ if and only if $x\in B$. 
With the above $B$ in hand, we denote
$$X_B:=X\cup Y_{m(X\cap B)},$$
and
$$L_B:=q_{X_B\bks B}\,\ell_B=:Q_B\,\ell_B$$
with $\ell_B$ a linear polynomial that we define in the sequel.
Assuming that we make sure that $\ell_B(v)\not=0$, it is clear that
$L_B(v)\not=0$. Our goal is to show, then, that $L_B(v')=0$, for
every $v'\in V_{\sk}\bks v$, and that $L_B\in\calP_{\sk}$.

Our first observation is that $Q_B$ above already vanishes at
``most" of the points in $V_{\sk}$. Indeed, each $B'\in A:=\{B'\in
\B_\sk: Q_B(\b(B'))\not =0\}$  has the following form:
\begin{equation}\label{bbprime}
B'=I'\cup J\cup\{y_{m(I)+1},\ldots,y_{m(I')}\},
\end{equation}
where  $J=B\cap Y$ and $I'=B'\cap X\subset B\cap X=I$ with
$\strangek(I')=\strangek(I)$. This is implied by the following
lemma.
\begin{lemma}\label{lem4cond}
$B'\in A$ if and only if the following four conditions hold:
\begin{itemize}
\item[(i)] $I'\subset I$,
\item[(ii)] $B'\cap Y_{m(I)}=J$,
\item[(iii)] $\strangek(I')=\strangek(I)$, and
\item[(iv)] $Y\cap B'\bks B=Y_{m(I')}\bks Y_{m(I)}$.
\end{itemize}
\end{lemma}

\begin{proof}[Proof of Lemma \ref{lem4cond}:]
First, it is easy to see that $B'\in A$ iff, with $v':=\b(B')$,
$q_{X\bks B}(v')\neq
0$ and $q_{Y_{m(I)}\bks B}(v')\neq 0$ iff $I'\subset I$
(Condition (i)) and $B'\cap Y_{m(I)}\subset B\cap Y=J$ (half of
Condition (ii)). We claim that this implies Condition (ii). In fact,
if $J':=B'\cap Y_{m(I)}$ is a proper subset of $J=B\bks I$,
then the maximal possible cardinality of $B'$ is
$$\#I'+\#J'+(m(I')-m(I))= \#I'+\#J'+\#I-\#I'+k(I')-k(I)\le\#I+\#J'<n,$$
since $I'\subset I$ and $\sk$ is solid. Therefore $B'\in A$ iff (i)
and (ii) hold.

Next we want to show that (iii) is implied by (i) and (ii). By (i)
and the fact that $\sk$ is solid, we have $\sk(I')\le \sk(I)$. Therefore,
$$m(I')-m(I)=\#(I\bks I')+\sk(I')-\sk(I)\le \#(I\bks I'),$$
with equality if and only if $\sk(I)=\sk(I')$. However, (ii) implies
that the set $Y_{m(I')}\bks Y_{m(I)}$ contains exactly $\#(I\bks
I')$ elements of $B'$ hence is of cardinality $\ge \#(I\bks I')$,
and (iii) thus follows.

Last, we show that condition (iv) is implied by the other
three: the argument in the previous paragraph shows that
$m(I')-m(I)=\#(I\bks I')$, and that $B'$ contains exactly $\#(I\bks I')$
vectors from
$Y_{m(I')}\bks Y_{m(I)}$, so  (iv) follows.
\end{proof}

So we have a bijection between $A$ and the subsets of $I'$ of $I$
that satisfy $\strangek(I')=\strangek(I)$. In that bijection, $I'$
is extended to $B'\in A$ via (\ref{bbprime}).

We now need to define $\ell_B$ in a way that it vanishes on $\b(A)\bks v$.
In view of the above bijection, we choose a {\it proper} subset
$I'$ of $I$ for which $\sk(I)=\sk(I')$, extend it to $B'$ as above, and
verify that our soon-to-be-defined $\ell_B$ vanishes at $v':=b(B')$.
Since $I'$ is a proper subset of $I$, it follows that $m(I')>m(I)$,
hence that the vector
$y':=y_{m(I)+1}$ lies in $B'$. Thus, $q_{y'}(v')=0$, and hence
the polynomial
$$Q_B\,q_{y'}$$
vanishes on $V_\sk\bks v$. However, this polynomial may not be in
$\calP_\sk$. We, therefore, write $q_{y'}$ as the sum
$\ell_B+(q_{y'}-\ell_B)$, with $\ell_B$ a linear polynomial that is
chosen so that: (i) $Q_B\,\ell_B\in \calP_\sk$, and (ii)
$q_{y'}-\ell_B$ vanishes on $\b(A)$. Condition (ii) will imply that
$\ell_B$ vanishes on $\b(A)\bks v$, hence $Q_B\ell_B$ is the
sought-for Lagrange polynomial.

To this end, we write $y'=\sum_{x\in B}a(x)x$, for some coefficients
$(a(x))_{x\in B}$,
and claim first that, if $x\in J$, or, alternatively,
if $x\in I'':=\{x\in I:\ \strangek(I\bks x)<\strangek(I)\}$, then,
in each case, $q_x$ vanishes on $\b(A)$. Once we prove
it, we define
$$\ell_B:=q_{y'}-\sum_{x\in I''\cup J}a(x)q_x,$$
and conclude that $\ell_B$ vanishes on $\b(A)\bks v$. Moreover, since
$p_{y'}=\sum_{x\in B}a(x)p_x$, we have that
$$\ell_B=\sum_{x\in I\bks I''}a(x)p_x-\lam_{y'}+\sum_{x\in I''\cup
J}a(x)\lam_x=\sum_{x\in I\bks I''}a(x)q_x-\lam_{y'}+\sum_{x\in
B}a(x)\lam_x.$$ Therefore, $Q_B\,\ell_B$ is a linear combination of
$Q_B$, and $Q_B\,q_x$, $x\in I\bks I''$.

Now, $Q_B$ itself lies in $\calP_\sk$: it is
the product of $q_Z$, $Z:=X\bks I$ by a polynomial $P$ of degree
$m(I)-(n-\#I)=\sk(I)$, hence lies in $\calP_\sk$ by Lemma \ref{pq}.
As to $Q_B\,q_x$, we can write it as the product $q_{Z\cup x}P$,
with $Z,P$ as above. Now, $X\bks (Z\cup x)=I\bks x$, and since we
assume that $\sk(I\bks x)=\sk (I)$, we still have that $\deg
P=\sk(I\bks x)$, hence by Lemma \ref{pq}, we have $q_{Z\cup x}\,P\in
\calP_\sk$. In conclusion, $Q_B\,\ell_B\in \calP_\sk$.

So, it remains  to show that
$q_x$ vanishes on $\b(A)$, whenever $x\in J\cup I''$. If $x\in J$,
then trivially, $q_x$ does so, since $J=B\bks I$ is a common subset
for all the bases in $A$ (cf.\ (ii) in Lemma \ref{lem4cond}).
Otherwise, $x\in I$, and $\strangek(I\bks x)<\strangek(I)$. Now, if
$q_x(\b(B'))\not=0$ for some $B'\in A$, then $x\not \in I':=B'\cap
X$. However, by property (i) of $A$, $I'\subset I$, and we conclude
that $I'\subset I\bks x$, and, since $\sk$ is solid, that
$\sk(I')\le \sk(I\bks x)<\sk(I)$, in contradiction to (ii) of Lemma
\ref{lem4cond}. So, $q_x$ vanishes on $\b(A)$ for every $x\in J\cup
I''$ and our proof is complete.
\endproof


\section{Incremental assignments}

Assuming that the assignment $\sk$ is solid, we have proved that
$$\dim \calP_\sk\ge \dim\calD_\sk=\#\B_\sk.$$
Moreover, the three different proofs for the inequality above that
were presented in \S3 show that:

\begin{corollary} Let $\sk$ be a  solid assignment and assume
that $\dim\calP_\sk=\#\B_\sk$. Then:
\begin{itemize}
\item $\calP_\sk$ and $\calD_\sk$ are dual to each other:
$$\calP_\sk\oplus \calJ_\sk=\Pi.$$
\item The homogeneous basis that was constructed in \S3.2 is a basis for
$\calP_\sk$.
\item The Lagrange basis that was constructed in \S3.3 is a basis
for $\calP_\sk$.\footnote{Since $\calD_\sk$ was proved to be equal
to $\Pi(V_\sk)$, then, once we know that $\calP_\sk$ is dual to
$\calD_\sk$, the existence of a Lagrange basis for $\calP_\sk$
follows. However, \S3.3 provides an explicit construction of that
basis.}
\end{itemize}
\end{corollary}

We will show in this section that the equality
\begin{equation}\label{pskeq}
\dim\calP_\sk=\#\B_\sk
\end{equation}
is valid once we assume $\sk$ to be (solid and) incremental.

\begin{example}[Continuation of Example \ref{eg3}]\label{eg4}
 We revisit the analysis made in Example \ref{eg3} of
$\calP_\sk$.  In the setup of that example, we already showed that
$\dim\calP_\sk\ge \#\B_\sk$, which must be the case since $\sk$ in
that example is solid.  Further, the example identifies exactly the
cases when equality holds: $\dim \mathcal{P}_{\strangek}= \dim
\mathcal{D}_{\sk}$ if and only if $\ell\in\{k,k+1\}$ and
$j\in\{k-1,k\}$. It is easy to check that these are exactly the
cases when the solid assignment $\sk$ is incremental. Thus, for the
simple setup of Example \ref{eg3}, the incrementality of $\sk$ is
{\it equivalent} to the equality (\ref{pskeq}).
\end{example}

In order to prove that (\ref{pskeq}) holds, we revisit the Lagrange
basis that was constructed in \S3.3, and that, so far, is only known
to be a basis for a subspace of $\calP_\sk$. We will show below
that, once $\sk$ is assumed to be incremental, a slightly simpler
version of this basis can be proved to span the entire $\calP_\sk$
space. To this end, we retain the notations from \S3.3, and in
particular the set
$$X_B:=X\cup Y_{m(B\cap X)}.$$
The Lagrange basis in \S3.3 was indexed by $\B_\sk$, with the basis
polynomial that corresponds to $B\in \B_\sk$ taking the form of the
product of
\begin{equation}
Q_B:=q_{X_B\bks B}
\end{equation}
and a carefully chosen linear polynomial $\ell_B$. It is
shown in the proof of Proposition \ref{conlang} that
the polynomials $Q_B$, $B\in \B_\sk$,
lie, each, in $\calP_{\strangek}$. The following theorem claims much
more:


\begin{theorem} \label{thminc} Assume that $\strangek$ is incremental.
Then the polynomials $(Q_B)_{B\in \B_{\strangek}}$ form a basis for
$\calP_{\strangek}$.
\end{theorem}

\begin{proof}[Proof of Theorem \ref{thminc}] Since we already know that
$\dim\calP_{\strangek}\ge \#\B_\sk$, and since we have exactly
$\#\B_{\strangek}$ functions in the polynomial set $(Q_B)_{B\in
\B_{\strangek}}$, we just need to prove that those polynomials span
$\calP_{\strangek}$. Let us denote by $\calQ$ their linear span. We
need to prove that, for every $Z\subset X$,
$$p_{X\bks Z}\Pi_{\sk(Z)}\subset
\calQ.$$  We first prove that for $I\in\I(X)$,
\begin{equation}\label{firstq}
q_{X\bks I}\Pi_{\strangek(I)}\subset \calQ.
\end{equation}
We prove this result by induction on $\#I$.
%
%
%
Denote $Y_I:=Y_{m(I)}\cup I$.  By our assumption on $Y$, the vectors
in the set $Y_I$ are in general position; also $\#(Y_I)=\sk(I)
+n-\#I+\#I= \sk(I)+n$. Therefore, the polynomials
$$q_W, \quad W\subset Y_I, \ \#W=\sk(I)$$
form a basis for $\Pi_{\sk(I)}$ (since they are linearly
independent: they form a Lagrange basis over the vertices of the
arrangement associated with $Y_I$). Therefore, once we show that
$$q_{X\bks I}q_W\in\calQ,$$
for every $W$ as above, we will conclude that  (\ref{firstq}) holds.
Now, if $W\subset Y$ (which is the only case if $I=\emptyset$)
then, with $J:=Y_{m(I)}\bks W$, we have
that $B:=I\cup J\in \B_\sk$, and that
$$q_{X\bks I}q_W=Q_B\in\calQ.$$
This completes the proof of (\ref{firstq}) for the initial case of the
induction ($I=\emptyset$). For all other $I$, we need to consider also
the case when $W\not\subset Y$. In that case, we
write
$$q_{X\bks I}q_W=q_{X\bks I} q_{W\cap X}q_{W\bks X},$$
and set $I':=I\bks (W\cap X)$. Then, $\#I-\#I'=\#(W\cap X)$, and we
conclude from the incremental property of $\sk$ that
$$\strangek(I')\ge \strangek(I)-\#(W\cap X).$$
On the other hand, $\#(W\bks X)=\#W-\#(W\cap
X)=\strangek(I)-\#(W\cap X)$. Consequently,
$$\deg q_{W\bks X}\le \strangek(I').$$
Thus,
$$q_{X\bks I}q_W=
q_{X\bks I'} q_{W\bks X}\in q_{X\bks I'}\Pi_{\sk(I')}
\subset\calQ,$$
with the last inclusion by the induction hypothesis (which we are allowed
to invoke since $I'$ is a proper subset of $I$).
This completes the proof of (\ref{firstq}).

Next, let $Z\subset X$, not necessarily independent. We want to show
that
\begin{equation}\label{secondq}
q_{X\bks Z}\Pi_{\sk(Z)}\subset\calQ.
\end{equation}
In order to prove the above, we consider $Z$ as a matroid, and let
$I\in \B(Z)\subset \I(X)$. Since $\Span I=\Span Z$ and $\sk$ is
solid, $\sk(I)=\sk(Z)$. Hence, by (\ref{firstq}),
$$q_{X\bks Z}q_{Z\bks I}\Pi_{\sk(Z)}=q_{X\bks I}\Pi_{\sk(I)}\subset \calQ.$$
This implies that
$$q_{X\bks Z}\Pi_{\sk(Z)}\Span\{q_{Z\bks I}: I\in \B(Z)\}\subset\calQ.$$
However, by \cite{DR}, the polynomials
$$q_{Z\bks I}, \quad I\in \B(Z)$$
form a basis for the central space $\calP(Z)$, hence we conclude
that
$$q_{X\bks Z}\Pi_{\sk(Z)}\calP(Z)\subset\calQ.$$
Since $\calP(Z)$ always contains the constants, we obtain
(\ref{secondq}).

Finally, we prove that, for every $Z\subset X$,
\begin{equation}\label{thirdq}
p_{X\bks Z}\Pi_{\sk(Z)}\subset \calQ.
\end{equation}
That will imply that $\calP_{\sk}\subset \calQ$, and will complete
the proof of the theorem. We prove (\ref{thirdq}) by induction on
$\#(X\bks Z)$, with the initial case $X=Z$ being trivial since for
this case there is no difference between $q_{X\bks Z}$ and $p_{X\bks
Z}$, hence (\ref{thirdq}) is implied here by (\ref{secondq}). Now,
assume that $Z\not =X$, and write
\begin{equation}\label{last}
p_{X\bks Z}=q_{X\bks Z}+\sum_{Z'}a(Z')p_{X\bks Z'},
\end{equation}
with $Z'$ ranging over all the proper supersets of $Z$. The
induction hypothesis implies that $p_{X\bks Z'}\Pi_{\sk(Z')}\subset
\calQ$, for each $Z'$ as above. Since $\strangek(Z')\ge
\strangek(Z)$, we have that
$$p_{X\bks Z'}\Pi_{\sk(Z)}\subset \calQ,\ Z'\supset Z,\ Z'\not=Z.$$
Since (\ref{secondq}) shows that $q_{X\bks Z}\Pi_{\sk(Z)}\subset
\calQ$, we conclude from (\ref{last}) that $p_{X\bks
Z}\Pi_{\sk(Z)}\subset \calQ$.
\end{proof}


\begin{thebibliography}{99}
\bibitem[AP]{AP}F. Ardila and A.
Postnikov, {\it Combinatorics and geometry of power ideals}, Trans.
Amer. Math. Soc. \textbf{362(8)} (2010), 4357-4384,
arXiv:0809.2143 MR 2608410 [math.CO]

\bibitem[B]{Biggs}N. Biggs, {\it Algebraic graph
theory}, Cambridge Mathematical Library. Cambridge University Press,
Cambridge, second edition, \textbf{1993}.


\bibitem[BDR]{BDR}C. de Boor, N. Dyn, and A. Ron, {\it On two polynomial
spaces associated with a box spline}, Pacific J. Math. {\bf 147} (1991),
249--267

\bibitem[BR90]{BR90}C. de Boor and A. Ron, {\it On multivariate polynomial interpolation},
 Constructive Approximation \textbf{6} (1990),
287-302.

\bibitem[BR91]{BR91}C. de Boor and A. Ron, {\it On ideals of finite codimension and applications
 to box Splines theory}, Journal of Mathematical Analysis and its Applications \textbf{158} (1991),
168-193.

\bibitem[BRS96]{BRS96}C. de Boor, A. Ron and Z. Shen, {\it On ascertaining inductively the dimension of the joint kernel of certain commuting
linear operators}, Advances in Applied Mathematics \textbf{17} (1996),
209-250.

\bibitem[DM]{DM}W. Dahmen and C. Micchelli, {\it On the local linear independence of translates of a box spline},
Studia Math. \textbf{82(3)} (1985),
243-263.

\bibitem[DR]{DR}N. Dyn and A. Ron, {\it Local approximation by certain
spaces of exponential polynomials, approximation order of exponential box
splines and related interpolation problems}, Trans.\   Amer.\ Math.\
Soc.\   \textbf{319} (1990), 381-404.

\bibitem[HR]{HR}O. Holtz and A.
Ron, {\it Zonotopal algebra}, Advances in  Mathematics (to
appear), arXiv:0708.2632 [math.CO]

\bibitem[HRX]{HRX}O. Holtz, A.
Ron, and Z. Xu, {\it Hierarchical zonotopal spaces}, Trans.\
Amer.\ Math.\  Soc., to appear, arXiv:0910.5543 [math.CO]

\bibitem[L]{L} M. Lenz, {\it Hierarchical Zonotopal Power Ideals},
arXiv:1011.1136

\bibitem[PS]{PS}A. Postnikov and B. Shapiro, {\it Trees, parking functions,
syzygies, and deformations of monomial ideals}, Trans.\
Amer.\ Math.\  Soc.\   \textbf{356} (2004), 3109-3142.

\bibitem[PSS]{PSS}A. Postnikov, B. Shapiro and M. Shapiro, {\it Algebras of
curvature forms on homogeneous manifolds}, Differential Topology,
Infinite-Dimensional Lie Algebras, and Applications: D.\ B.\ Fuchs 60th
Anniversary Collection, AMS Translations, Series 2 \textbf{194} (1999),
227-235.




\end{thebibliography}
\end{document}